\documentclass[a4paper,12pt]{article}

\usepackage{latexsym}
\usepackage{mathrsfs}
\usepackage{amsfonts,amsmath,amssymb}
\usepackage{indentfirst}
\usepackage{subeqnarray}
\usepackage{verbatim}
\usepackage[pdftex]{graphicx}
\usepackage{epstopdf}
\usepackage{stmaryrd}
\usepackage{multirow}
\usepackage{diagbox}
\usepackage{subcaption}

\newcommand{\bx}{\mbox{\boldmath{$x$}}}

\newcommand{\bb}{\mbox{\boldmath{$b$}}}

\newcommand{\bn}{\mbox{\boldmath{$n$}}}

\newcommand{\bW}{\mbox{\boldmath{$W$}}}

\newcommand{\dx}{\mathrm{d}x}
\newcommand{\ds}{\mathrm{d}s}
\newcommand{\dt}{\mathrm{d}t}

\newtheorem{theorem}{Theorem}[section]
\newtheorem{lemma}[theorem]{Lemma}

\newtheorem{example}[theorem]{Example}

\newtheorem{remark}[theorem]{Remark}
\numberwithin{equation}{section}

\newenvironment{proof}[1][Proof]{\textbf{#1.} }
{\ \rule{0.75em}{0.75em}\smallskip}

\usepackage[pdftex,unicode,colorlinks,linkcolor=blue]{hyperref}

\begin{document}

\begin{center}
\Large\bf Local Randomized Neural Networks with Discontinuous Galerkin Methods for Diffusive-Viscous Wave Equation
\end{center}

\begin{center}
Jingbo Sun\footnote{School of Mathematics and Statistics, Xi'an Jiaotong University, Xi'an, Shaanxi 710049, P.R. China. E-mail: {\tt jingbosun@stu.xjtu.edu.cn}.},
\quad 
Fei Wang\footnote{School of Mathematics and Statistics, Xi’an Jiaotong University, Xi’an, Shaanxi 710049, China. The work of this author was partially supported by the National Natural Science Foundation of China (Grant No. 12171383). Email: {\tt feiwang.xjtu@xjtu.edu.cn}.}
\end{center}

\medskip
\begin{quote}
  {\bf Abstract.} 
The diffusive-viscous wave equation is an advancement in wave equation theory, as it accounts for both diffusion and viscosity effects. This has a wide range of applications in geophysics, such as the attenuation of seismic waves in fluid-saturated solids and frequency-dependent phenomena in porous media. Therefore, the development of an efficient numerical method for the equation is of both theoretical and practical importance. Recently, local randomized neural networks with discontinuous Galerkin (LRNN-DG) methods have been introduced in \cite{Sun2022lrnndg} to solve elliptic and parabolic equations. Numerical examples suggest that LRNN-DG can achieve high accuracy, and can handle time-dependent problems naturally and efficiently by using a space-time framework. In this paper, we develop LRNN-DG methods for solving the diffusive-viscous wave equation and present numerical experiments with several cases. The numerical results show that the proposed methods can solve the diffusive-viscous wave equation more accurately with less computing costs than traditional methods.
 
\end{quote}

\medskip

\section{Introduction}

Petroleum exploration is the process of using various methods to obtain geological information about the underground, identify favorable areas and traps for oil and gas accumulation, estimate the size and production capacity of oil and gas fields, and discover and verify potential oil and gas resources. Numerical simulation of seismic waves can reduce costs and improve the success rate of petroleum exploration. The diffusive-viscous wave equation (DVWE) was proposed in \cite{Goloshubin2000dvwe, Korneev2004dvwe} to fully describe the frequency-dependent reflection and fluid-saturation in porous media by introducing frictional and viscous damping terms based on wave equations, which enables the detection of fluid-saturated layers. So effective numerical methods for solving the DVWE are essential for geological exploration.

In the past ten years, the diffusive-viscous wave equations have been studied and applied for various practical issues, see e.g. \cite{Quintal2007dvwe,Chen2013dvwe, Zhao2014dvwe}. Han et al. (\cite{Han2020well}) established a well-posedness theory of solution for a general initial boundary value problem of the diffusive-viscous wave equation. The first numerical analysis of the finite element method for the diffusive-viscous wave equation was presented in \cite{Han2021FEM}. Recently, the interior penalty discontinuous Galerkin (DG) method and local DG method were applied for solving the diffusive-viscous wave equation (\cite{Zhang2023dg,Ling2023LDG}). Discontinuous Galerkin method is a high-order accurate, robust, and flexible method that can handle complex geometries and boundary conditions, and incorporate physics-informed numerical fluxes.

With the advancement of neural networks, there has been an increasing number of studies dedicated to the numerical solution of partial differential equations (PDEs) based on neural networks in recent years. To name a few, these include the Deep Ritz Method (\cite{Ee2018deepRitz}), Deep Galerkin Method (\cite{Sirignano2018DGM}), Physical Informed Neural Networks (\cite{Raissi2019PINN}), Weak Adversarial Networks (\cite{Zang2020adversarialnns}), and Deep Nitsche Method (\cite{Liao2019deepRitzboundary}). These neural network-based methods construct loss functions involving PDEs or weak formulations, and rely on optimization solvers to train the networks.

Dong and Li proposed an approach called Local Extreme Learning Machine (ELM) and Domain Decomposition to solve PDEs in \cite{Dong2021locELM}. ELM is a special case of Randomized Neural Networks (RNNs, \cite{Igelnik1995RNN1, Igelnik1999RNN2, Pao1992RNN3, Pao1994RNN4}), where parameters are randomly assigned to links between hidden layers and then fixed during training. The links between the last hidden layer and output layer are determined using a least-squares method, thus avoiding the need for an optimization solver. In addition, Shang et al. proposed Deep Petrov-Galerkin Methods for solving PDEs by combining RNNs and Petrov-Galerkin formulation in \cite{Shang2022DeepPetrov}. Sun et al. proposed Local Randomized Neural Network with DG (LRNN-DG) Methods for solving PDEs in \cite{Sun2022lrnndg}. Numerical experiments demonstrate that LRNN-DG approaches can attain accuracy comparable to conventional methods such as Finite Element Methods and DG methods, but with a reduced number of degrees of freedom. Moreover, LRNN-DG methods can naturally adopt a space-time approach, enabling them to solve time-dependent problems without requiring time discretization and avoiding error accumulations.

In this paper, we develop LRNN-DG methods to solve the diffusive-viscous wave equation in the space-time approach. These methods use RNNs to construct space-time basis functions on each subdomain and apply DG formulation to couple them. The paper is structured as follows. Section \ref{sec:weak} describes the neural network architecture. Section \ref{dgformula} introduces three new discontinuous Galerkin formulations with RNNs for the space-time approximation of the diffusive-viscous wave equation. Section \ref{numericalex} demonstrates the performance of the proposed LRNN-DG methods with several numerical examples. Section \ref{summary} concludes the paper with some remarks.

%%%%%%%%%%%%%%%%%%%%%%%%%%%%%
\section{Neural network structure}
\label{sec:weak}

In this section, we introduce the neural network structure used in this work. 
A common type of deep learning model is the fully connected neural network. In this model, we omit the bias term of the output layer and use the following structure:
\begin{subequations}\label{fnn_structure}
\begin{align}
\mathcal{U}(\bx) &=\bW^{(L+1)} ( N^{(L)} \circ \cdots N^{(2)} \circ N^{(1)}(\bx)), \\
N^{(i)}(\bx) &=\sigma(\bW^{(i)}\bx+\bb^{(i)}),\ i= 1,2,\cdots, L,
\end{align}
\end{subequations}
where $\bx \in \Omega$ is the input vector, $N^{(i)}$ is the i-th hidden layer with the weight matrix $\bW^{(i)}$ and the bias vector $\bb^{(i)}$, $\sigma$ is a nonlinear activation function, $L$ is the numbers of the hidden layers, also called the depth of the network, and $\bW^{(L+1)}$ is the weight matrix of the output layer. The output layer performs a linear transformation on the final hidden layer. The set of all functions that can be represented by this network is denoted by
\begin{equation*}
\mathcal{M}(\theta,L,\Omega)=\{\mathcal{U}(\bx)= \bW^{(L+1)} (N^{(L)} \circ \cdots \circ N^{(1)}(\bx)), \bx\in\Omega\},
\end{equation*}
where $\theta = \{\bW^{(L+1)},(\bW^{(l)} , \bb^{(l)})^{L}_{l=1}\}$.

\begin{figure}[thpb] 
  \centering  
  \includegraphics[width=0.6\textwidth]{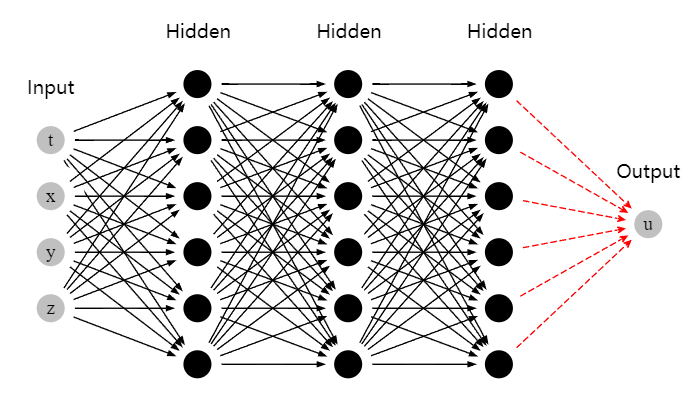}  
  \caption{The structure of neural networks.}
  \label{nn_structure}
\end{figure}

Unlike a function space, $\mathcal{M}(\theta,L,\Omega)$ is a set of functions that are obtained by optimizing the network parameters using methods such as stochastic gradient descent type algorithms. These methods are often slow and inaccurate. To overcome this limitation, we use the randomized neural network approach in this paper. A randomized neural network is a neural network where some parameters are randomly assigned and fixed. This approach does not require backpropagation to train the network. Instead, we randomly assign and fix the values of the hidden layer parameters, and then we use the least-squares method to find the output layer weights. As an example, Figure \ref{nn_structure} shows the structure of our model with 3-d space domain and 1-d time domain, here this model has 3 hidden layers and 6 neurons in each layer. We note black solid lines represent weights and biases that are randomly assigned and fixed and red dashed lines represent weights computed by the least-squares method.
The function space for RNNs can be written as 
\begin{equation}\label{rnn_space}
\mathcal{M}_{RNN} (D)=\left\{\mathcal{U}(\alpha, \theta, \bx) = \sum^M_{j=1} \alpha^D_j \phi^D (\theta_j, \bx) : \bx \in D\right\},
\end{equation}
where $D\subset\Omega$ is the domain, $\phi^D (\theta_j, \bx)$ is the output of the hidden layer with weight and bias $\theta_j$, $\alpha^D_j$ is the weight of the output layer and $M$ is the number of neurons in the last hidden layer. For simplicity, we denote $\phi^D (\theta_j, \bx)$ as $\phi^D_j(\bx)$ in the rest of the paper.

One example of randomized neural networks is the extreme learning machine (ELM) (\cite{Huang2006ELMtheorandapp}). It has been shown in \cite{Liu2014ELMfeasible} that ELM can achieve a similar generalization performance as regular neural networks if the activation functions and the initialization strategies for the fixed parameters are chosen appropriately.
This approach has shown good performance for solving partial differential equation problems in \cite{Dong2021locELM,Shang2022DeepPetrov,Sun2022lrnndg}, as it has advantages such as shorter training time and better accuracy. However, for complicated problems, a single neural network may lose accuracy, and we need to apply randomized neural networks on each subdomain. Therefore, in this work, we follow the ideas in \cite{Sun2022lrnndg} and develop local randomized neural networks with discontinuous Galerkin methods to solve the diffusive-viscous wave equation.

\section{LRNN-DG methods for DVWE}
\label{dgformula}

In the previous section, we mentioned that a single neural network may not be able to approximate complex problems well. Therefore, we need to split a domain into several subdomains and use a randomized neural network on each one. But how do we connect these local randomized neural networks as a global solution? There are two main methods: one is to impose continuity conditions on the interfaces between subdomains as in \cite{Dong2021locELM}, and the other is to apply discontinuous Galerkin (DG) formulation to join them together as in \cite{Sun2022lrnndg}. In this section, we develop local randomized neural networks with three DG schemes for solving the diffusive-viscous wave equation.

\subsection{LRNN-DG formulation}

Let us consider the diffusive-viscous wave equation with mixed boundary conditions of Dirichlet, Neumann, and Robin types. The initial-boundary value problem can be written as
\begin{subequations}\label{dvwe_strong}
\begin{align}
\frac{\partial^2 u}{\partial t^2} +\gamma\frac{\partial u}{\partial t} - \frac{\partial}{\partial t}{\rm div}(\eta \nabla u) - {\rm div}(\xi^2 \nabla u) &= f \quad &{\rm in}\ \Sigma, \label{strongeq}\\
u &= g_D \quad &{\rm on}\ I \times \Gamma_D, \label{strongboundcd}\\
\frac{\partial u}{\partial \boldsymbol{n} } &= g_N \quad &{\rm on}\ I \times \Gamma_N,\label{strongboundcn}\\
\frac{\partial u}{\partial \boldsymbol{n} } + \kappa u&= g_R \quad &{\rm on}\ I \times \Gamma_R,\label{strongboundcr}\\
{ u} &= u_0\quad &{\rm on}\ \{0\}\times\Omega,\label{stronginit1} \\
\frac{\partial u}{\partial t} &= w_0\quad &{\rm on}\ \{0\}\times\Omega,\label{stronginit2}
\end{align}
\end{subequations}
where $\Omega\subset\mathbb{R}^d$ $(d=1,2,3)$ is the spatial domain, and $I=(0,T)$ is the temporal domain. 
We denote $\Sigma = I\times\Omega$ as the space-time domain and $\partial \Omega$ as the boundary of $\Omega$. The boundary $\partial \Omega$ is partitioned into disjoint parts $\partial \Omega = {\overline {\Gamma_D}}\times{\overline {\Gamma_N}}\times{\overline {\Gamma_R}}$, where ${ {\Gamma_D}}$, ${ {\Gamma_N}}$, ${ {\Gamma_R}}$ are the Dirichlet, Neumann, and Robin boundaries, respectively. One or two of these parts can be empty. The functions $g_D$, $g_N$, and $g_R$ are the prescribed boundary values for each type of condition. The vector $\boldsymbol{n}$ is the unit outward normal to $\partial \Omega$. The coefficients $\gamma = \gamma(\bx)>0$, $\eta = \eta(\bx)>0$ and $\xi = \xi(\bx)>0$ represent the diffusive attenuation, the viscous attenuation and the wave propagation speed in the non-dispersive medium, respectively. $\kappa = \kappa(\bx)>0$ is the coefficient for the Robin boundary condition. The function $f(t,\bx)$ is the source term.

Before we proceed, we introduce some notation for dividing the space-time domain into subdomains. We partition the time interval $I$ into $N_t$ sub-intervals $\mathcal{D}_\tau =\{I_i =(t_{i-1},t_i), 0=t_0< t_1<\cdots<t_{N_t}=T\}$ with $\tau = \max\limits_{I_i\in {\cal D}_\tau}\{{\rm diam}(I_i)\}$, let ${\cal P}_\tau=\{t_i,i=0,\cdots,{N_t}\}$ be the set of time nodes, and $\mathcal{P}_\tau^{i} = \mathcal{P}_\tau \backslash \{t_0,t_{N_t}\}$ is the set of all interior points. We also decompose the spatial domain $\bar{\Omega}$ into a mesh $\{ \mathcal{T}_h \}$, where $h = \max_{K\in {\cal T}_h}\{{\rm diam}(K)\}$ and $\mathcal{T}_h$ consists of $N_s$ elements. We denote by $\mathcal{E}_h$ the union of all the faces (edges) of the spatial mesh, by $\mathcal{E}_h^i$ the set of the interior faces (edges) and by $\mathcal{E}^\partial_h = \mathcal{E}_h\backslash {\cal E}_h^i = {\cal E}_h^D\cup{\cal E}_h^N\cup{\cal E}_h^R$ the set of the boundary (faces) edges, where $ {\cal E}_h^D$, ${\cal E}_h^N$, ${\cal E}_h^R$ are (faces) edges corresponding to Dirichlet, Neumann, and Robin boundaries respectively. Then we use $\{\mathcal{D}_\tau\times\mathcal{T}_h\}$ to denote the decomposition of the space-time domain $\bar\Sigma$, which has $N_e=|\mathcal{D}_\tau\times\mathcal{T}_h|=N_t N_s$ elements. 

For any two adjacent elements $\sigma^+_h = I_i\times K^+$ and $\sigma^-_h= I_i\times K^-$ sharing a common spatial face $f_h$, we define $\boldsymbol{n}^{\pm}=\boldsymbol{n}|_{\partial K^{\pm}}$ as the unit outward normal vectors
on $\partial K^{\pm}$. For a scalar function $v$ and a
vector function $\boldsymbol{q}$, we write $v^{\pm}=v|_{\partial \sigma_h^{\pm}}$ and $\boldsymbol{q}^\pm=\boldsymbol{q}|_{\partial \sigma_h^{\pm}}$. We also define the averages $\{ \cdot \}$ and the jumps $\llbracket \cdot \rrbracket$, $[ \cdot ]$ on $f_h \in (\mathcal{D}_\tau\times\mathcal{E}_h^{i})$ by
\begin{subequations}
\begin{align*}
\{ v \} = \frac{1}{2} (v^+ + v^-),& \quad \llbracket v\rrbracket = v^+ \boldsymbol{n}^+ + v^- \boldsymbol{n}^-, \\
\{ \boldsymbol{q} \} = \frac{1}{2}(\boldsymbol{q}^+ +\boldsymbol{q}^-),& \quad
[\boldsymbol{q}] = \boldsymbol{q}^+ \cdot \boldsymbol{n}^+ +\boldsymbol{q}^- \cdot \boldsymbol{n}^-.
\end{align*}
If $f_h \in (\mathcal{D}_\tau\times \mathcal{E}_h^{\partial})$, we set
\begin{subequations}
\begin{align*}
\llbracket v\rrbracket = v \boldsymbol{n},\quad
\{ \boldsymbol{q} \} = \boldsymbol{q},
\end{align*}
\end{subequations}
where $\boldsymbol{n}$ is the unit outward normal vector on $\partial\Omega$.

For any two adjacent elements $\sigma^+_\tau = I_{i+1}\times K$ and $\sigma^-_\tau= I_i\times K$ sharing a common temporal face $f_\tau \in\{t_i\}\times \mathcal{T}_h$, we write $w(t^{\pm}_i,\bx)=w(t_i,\bx)|_{\partial \sigma_\tau^{\pm}}$ for a scalar function $w$ and define the averages $\{ \cdot \}$ and the jumps $[ \cdot ]$ on $f_\tau \in (\mathcal{P}_\tau^{i}\times \mathcal{T}_h)$ by
\begin{align*}
\{ w(t_i,\bx) \} = \frac{1}{2} \left(w(t_i^+,\bx) + w(t_i^-,\bx)\right),& \quad [w(t_i,\bx)] =w(t_i^+,\bx) - w(t_i^-,\bx).
\end{align*}
\end{subequations}
If $f_\tau \in (\mathcal{P}_\tau^{\partial}\times \mathcal{T}_h)$, we set
\begin{subequations}
\begin{align*}
[w(t_0,\bx)] = - w(t_0,\bx),\quad [w(t_{N_t},\bx)] =  w(t_{N_t},\bx),\quad
\{ w(t,\bx) \} = w(t,\bx).
\end{align*}
\end{subequations}

To derive the formulation of the LRNN-DG method for DVWE, we use the following identities:
\begin{align}
\int_K \nabla v \cdot \boldsymbol{q}\,\dx &= -\int_K v~(\nabla\cdot \boldsymbol{q}) \,\dx + \int_{\partial K} v\,\boldsymbol{q}\cdot\bn_K\,\ds,\label{ibps}\\
\sum_{K\in \mathcal{T}_h} \int_{\partial K} v \boldsymbol{q} \cdot \boldsymbol{n}_K \, \ds
&= \int_{{\cal E}_h} \llbracket v\rrbracket \cdot \{\boldsymbol{q}\}\, \ds
+ \int_{{\cal E}_h^i} \{v\} \cdot [\boldsymbol{q}]\, \ds,\label{iden_dg}
\\
\sum_{I_i\in \mathcal{D}_\tau} (v w)|^{t_i}_{t_{i-1}}
&= \sum^{N_t}_{i=0} \,[v(t_i,\bx)]  \{w(t_i,\bx)\}
+ \sum^{N_t -1}_{i=1} \,\{v(t_i,\bx)\} [w(t_i,\bx)].\label{iden_tdg}
\end{align}

We introduce the LRNN-DG function spaces based on the decomposition of the space-time domain and the local randomized neural network function spaces \eqref{rnn_space}: 
\begin{align}
V^\tau_h&=\{v_h \in L^2(\Sigma): \;v^\tau_h |_\sigma \in \mathcal{M}_{RNN}(\sigma)\quad \forall\,\sigma\in\mathcal{D}_\tau\times\mathcal{T}_h \},\label{scalspace}\\
\boldsymbol{Q}^\tau_h&=\{\boldsymbol{q}^\tau_h \in [L^2(\Sigma)]^d: \;\boldsymbol{q}^\tau_h |_\sigma \in \left[\mathcal{M}_{RNN}(\sigma)\right]^d\quad \forall\,\sigma\in\mathcal{D}_\tau\times\mathcal{T}_h \}\label{vecspace}.
\end{align}
Instead of using polynomial functions for the trial and test functions as in conventional DG methods, we use the space $V^\tau_h$ consisting of many random basis functions from neural networks. This allows us to leverage the strong approximation ability of neural networks to approximate the DVWE in a space-time approach.

To derive the DG formulation, let us rewrite Eq. \eqref{strongeq} as the following system of equations,
\begin{align}
{\boldsymbol{p}} = \nabla u \quad {\rm in} \ \Sigma, \label{str_mix1}\\
w = \frac{\partial u}{\partial t}\quad {\rm in} \ \Sigma,\label{str_mix2}\\
\frac{\partial w}{\partial t} +\gamma w - \nabla \cdot (\eta \nabla w) - \nabla \cdot ({\xi^2 \boldsymbol{p}}) = f \quad {\rm in} \ \Sigma.\label{str_mix3}
\end{align}
For any subdomain $\sigma\in\mathcal{D}_\tau\times\mathcal{T}_h$, we multiply test functions $v^\tau_h,b^\tau_h \in V^\tau_h$ and $\boldsymbol{q}^\tau_h\in\boldsymbol{Q}^\tau_h$ on both sides of Eqs. \eqref{str_mix1}-\eqref{str_mix3}, respectively, and integrate over $\sigma$. Then we apply the integration by parts formula and use numerical traces $\widehat{u^\tau_h}$, $\widehat{\boldsymbol{p}^\tau_h}$ and $\widehat{\nabla w^\tau_h}$ to approximate $u$, $\boldsymbol{p}$ and $\nabla w$ on the spatial faces $f\in \mathcal{D}_\tau\times\mathcal{E}_h$ and use numerical traces $\widetilde{u^\tau_h}$ and $\widetilde{\boldsymbol{p}^\tau_h}$ to approximate $u$ and $\boldsymbol{p}$ on the temporal faces $f\in \mathcal{P}_\tau\times\mathcal{T}_h$. This leads to
\begin{align*}
&\int_{\sigma} {\boldsymbol{p}^\tau_h} \cdot {\boldsymbol{q}^\tau_h} dx dt = - \int_{\sigma} u^\tau_h \nabla \cdot {\boldsymbol{q}^\tau_h} \dx\dt + \int_{I_i \times \partial K} \widehat{u^\tau_h} \boldsymbol{q}^\tau_h\cdot\boldsymbol{n}_K \ds\dt,\\
&\int_{\sigma} w^\tau_h b^\tau_h \dx\dt = \int_{K}(\widetilde{u^\tau_h} b^\tau_h) |^{t_i}_{t_{i-1}}\dx -\int_{\sigma} u^\tau_h \frac{\partial b^\tau_h}{\partial t} \dx\dt,\\
& \int_{K}(\widetilde{w^\tau_h} v^\tau_h) |^{t_i}_{t_{i-1}}\dx -\int_{\sigma} w^\tau_h\frac{\partial v^\tau_h}{\partial t} \dx \dt + \int_{\sigma} \gamma w^\tau_h v^\tau_h \dx\dt -\int_{I_i\times\partial K} \eta \widehat{\nabla w^\tau_h} \cdot \boldsymbol{n}_K v^\tau_h \ds\dt \\
&+ \int_{\sigma}\eta {\nabla w^\tau_h} \cdot \nabla v^\tau_h \dx\dt -\int_{I_i\times\partial K} \xi^2 \widehat{\boldsymbol{p}^\tau_h} \cdot \boldsymbol{n} v^\tau_h \ds\dt +\int_{\sigma} \xi^2 \boldsymbol{p}^\tau_h \cdot \nabla v^\tau_h\dx\dt=  \int_{\sigma} f v^\tau_h \dx \dt.
\end{align*}

Next, we sum over all the elements and apply integration by parts and Eq. \eqref{iden_dg}. To eliminate 
$\boldsymbol{p}^\tau_h$, $w^\tau_h$, we choose $\boldsymbol{q}^\tau_h = \nabla_h v^\tau_h$, $b^\tau_h = \frac{\partial v^\tau_h}{\partial t}$ 
and let $w^\tau_h = \frac{\partial_\tau u^\tau_h}{\partial_\tau t}$. Then a combination gives us,

\begin{equation*}
\begin{aligned}
&- \int_{\Sigma}\frac{\partial_\tau u^\tau_h}{\partial_\tau t} \frac{\partial_\tau v^\tau_h}{\partial_\tau t} \dx \dt -\sum^{N_t}_{i=0}\int_{\mathcal{T}_h}[ \widetilde{u^\tau_h}(t_i,\bx)-u^\tau_h(t_i,\bx) ] \left\{ \frac{\partial_\tau v^\tau_h}{\partial_\tau t}(t_i,\bx)\right\}\dx \\
&-\sum^{N_t-1}_{i=1}\int_{\mathcal{T}_h}\{ \widetilde{u^\tau_h}(t_i,\bx)-u^\tau_h(t_i,\bx) \}  \left[ \frac{\partial_\tau v^\tau_h}{\partial_\tau t}(t_i,\bx)\right]\dx
+\int_{\Sigma}\gamma \frac{\partial_\tau u^\tau_h}{\partial_\tau t} v^\tau_h \dx \dt\\
&+\int_{\Sigma}\eta \nabla_h \left(\frac{\partial_\tau u^\tau_h}{\partial_\tau t}\right)\cdot\nabla_h v^\tau_h \dx \dt - \int_{\mathcal{D}_\tau\times\mathcal{E}^{}_h} \eta \llbracket v^\tau_h\rrbracket \cdot\{\widehat{\nabla_h w^\tau_h}\} \ds\dt - \int_{\mathcal{D}_\tau\times\mathcal{E}^{i}_h} \eta \left\{ v^\tau_h\right\} \cdot\left[\widehat{\nabla_h w^\tau_h}\right] \ds\dt\\
&+ \int_{\Sigma}\xi^2\nabla_h u^\tau_h \cdot\nabla_h v^\tau_h \dx \dt +\int_{\mathcal{D}_\tau\times\mathcal{E}^{}_h}\xi^2\llbracket \widehat{u^\tau_h}-u^\tau_h \rrbracket \cdot \left\{ \nabla_h v^\tau_h\right\}\ds\dt 
+\int_{\mathcal{D}_\tau\times\mathcal{E}^{i}_h}\xi^2\{ \widehat{u^\tau_h}-u^\tau_h \} \cdot [ \nabla_h v^\tau_h ]\ds\dt \\
&-\int_{\mathcal{D}_\tau\times\mathcal{E}^{}_h}\xi^2\llbracket v^\tau_h \rrbracket \cdot \{ \widehat{\boldsymbol{p}^\tau_h}\}\ds\dt -\int_{\mathcal{D}_\tau\times\mathcal{E}^{i}_h}\xi^2\{ v^\tau_h \} \cdot [  \widehat{\boldsymbol{p}^\tau_h} ]\ds\dt 
+\sum^{N_t}_{i=0}\int_{\mathcal{T}_h}[ v^\tau_h(t_i,x) ]  \{\widetilde{w^\tau_h}(t_i,\bx)\}\dx \\
&+\sum^{N_t-1}_{i=1}\int_{\mathcal{T}_h}\{ v^\tau_h(t_i,\bx) \}  [ \widetilde{w^\tau_h}(t_i,\bx)]\dx
=\int_{\Sigma}fv^\tau_h\dx\dt,
\end{aligned}
\end{equation*}
where $\nabla_hv_h^\tau$ is the broken gradient of $v_h^\tau$ with respect to the mesh $\mathcal{T}_h$ and $\frac{\partial_\tau v^\tau_h}{\partial_\tau t}$ is the broken partial derivative of $v_h^\tau$ with respect to the partition $\mathcal{D}_\tau$.

The consistency and stability of the DG scheme depend on the appropriate choices of numerical fluxes. We choose the following numerical fluxes in this paper
\begin{subequations}
\begin{align*}
\widetilde{u^\tau_h}(t_i,\bx) = \{ u^\tau_h (t_i,\bx)\},\ \ \widetilde{w^\tau_h}(t_i,\bx) = \left\{ \frac{\partial_\tau u^\tau_h}{\partial_\tau t}(t_i,\bx) \right\} - \beta_1 [ u^\tau_h(t_i,\bx)] \quad &on \ f\in\mathcal{P}^i_\tau\times\mathcal{T}_h,\\
\widetilde{u^\tau_h}(t_0,\bx) = u_0,\ \ \widetilde{w^\tau_h}(t_0,\bx)= -w_0 - \beta_1 (u_0 -u^\tau_h(t_0,\bx)) \quad &on \ f\in\{t_0\}\times\mathcal{T}_h,\\
\widetilde{u^\tau_h}(t_{N_t},\bx) = u^\tau_h(t_{N_t},\bx),\ \ \widetilde{w^\tau_h}(t_{N_t},\bx) = \frac{\partial_\tau u^\tau_h}{\partial_\tau t}(t_{N_t},\bx) \quad &on \ f\in\{t_{N_t}\}\times\mathcal{T}_h,\\
\widehat{u^\tau_h} = \{ u^\tau_h \},\ \ \widehat{\boldsymbol{p}^\tau_h} = \{ \nabla_h u^\tau_h \} - \beta_2 \llbracket u^\tau_h\rrbracket \quad &on \ f\in\mathcal{D}_\tau\times\mathcal{E}^i_h,\\
\widehat{u^\tau_h} =g,\ \ \widehat{\boldsymbol{p}^\tau_h} = \nabla_h u^\tau_h - \beta_2 (u-g_D) \boldsymbol{n} \quad &on \ f\in\mathcal{D}_\tau\times\mathcal{E}^D_h,\\
\widehat{u^\tau_h} =u^\tau_h,\ \ \widehat{\boldsymbol{p}^\tau_h} \cdot \boldsymbol{n}   = g_N \quad &on \ f\in \mathcal{D}_\tau\times\mathcal{E}^N_h,\\
\widehat{u^\tau_h} =u^\tau_h,\ \ \widehat{\boldsymbol{p}^\tau_h} \cdot \boldsymbol{n} + \kappa u^\tau_h = g_R \quad &on \ f\in \mathcal{D}_\tau\times\mathcal{E}^R_h,\\
\widehat{\nabla_h w^\tau_h} =\left\{\nabla_h \left( \frac{\partial_\tau u^\tau_h}{\partial_\tau t}\right)\right\} \quad &on \ f\in\mathcal{D}_\tau\times\mathcal{E}_h,%\\
%\widetilde{u^\tau_h}(t_i,\bx) = \{ u^\tau_h (t_i,\bx)\},\ \ \widetilde{w^\tau_h}(t_i,\bx) = \left\{ \frac{\partial_\tau u^\tau_h}{\partial_\tau t}(t_i,\bx) \right\} - \beta_1 [ u^\tau_h(t_i,\bx)] \quad &on \ f\in\mathcal{P}^i_\tau\times\mathcal{T}_h,\\
%\widetilde{u^\tau_h}(t_0,\bx) = u_0,\ \ \widetilde{w^\tau_h}(t_0,\bx)= -w_0 - \beta_1 (u_0 -u^\tau_h(t_0,\bx)) \quad &on \ f\in\{t_0\}\times\mathcal{T}_h,\\
%\widetilde{u^\tau_h}(t_{N_t},\bx) = u^\tau_h(t_{N_t},\bx),\ \ \widetilde{w^\tau_h}(t_{N_t},\bx) = \frac{\partial_\tau u^\tau_h}{\partial_\tau t}(t_{N_t},\bx) \quad &on \ f\in\{t_{N_t}\}\times\mathcal{T}_h,
\end{align*}
\end{subequations}
where $\beta_1$ and $\beta_2$ are penalty parameters that may vary depending on the face $f$.
Then the LRNN-DG scheme for solving the diffusive-viscous wave equation \eqref{dvwe_strong} is: Find $u^\tau_h \in V_h^\tau$ such that
\begin{equation}\label{timec0for}
\begin{aligned}
B_{h\tau}(u^\tau_h,v^\tau_h) = l(v^\tau_h) \quad \forall v^\tau_h\in V^\tau_h,
\end{aligned}
\end{equation}
where
\begin{align}\label{lrnndg_B}
B_{h\tau}(u^\tau_h,v^\tau_h)=
&- \int_{\Sigma}\frac{\partial_\tau u^\tau_h}{\partial_\tau t} \frac{\partial_\tau v^\tau_h}{\partial_\tau t} \dx \dt + \int_{\Sigma}\gamma \frac{\partial_\tau u^\tau_h}{\partial_\tau t} v^\tau_h \dx \dt +\int_{\Sigma}\eta \nabla_h \left(\frac{\partial_\tau u^\tau_h}{\partial_\tau t}\right)\cdot\nabla_h v^\tau_h \dx \dt \nonumber\\ 
&+ \int_{\Sigma}\xi^2 \nabla_h u^\tau_h \cdot\nabla_h v^\tau_h \dx \dt +\sum^{N_t -1}_{i=0}\int_{\mathcal{T}_h}[ u^\tau_h(t_i,\bx) ] \left\{ \frac{\partial_\tau v^\tau_h}{\partial_\tau t}(t_i,\bx)\right\}\dx \nonumber\\ 
&+\sum^{N_t}_{i=1}\int_{\mathcal{T}_h}[ v^\tau_h(t_i,\bx)] \left\{\frac{\partial_\tau u^\tau_h}{\partial_\tau t}(t_i,\bx)\right\}\dx 
-\sum^{N_t -1}_{i=0}\int_{\mathcal{T}_h}\beta_1[ u^\tau_h(t_i,\bx) ]  [ v^\tau_h(t_i,\bx) ]\dx  \nonumber\\ 
&-\int_{\mathcal{D}_\tau \times\mathcal{E}^{}_h} \eta \llbracket v^\tau_h\rrbracket\cdot \left\{ \nabla_{h} \left(\frac{\partial_\tau u^\tau_h}{\partial_\tau t}\right) \right\} \ds\dt - \int_{\mathcal{D}_\tau \times \left(\mathcal{E}^{i}_h\cup\mathcal{E}^{D}_h\right)} \xi^2 \llbracket u^\tau_h\rrbracket\cdot \{ \nabla_{h}v^\tau_h \} \ds\dt \nonumber\\ 
&- \int_{\mathcal{D}_\tau \times \left(\mathcal{E}^{i}_h\cup\mathcal{E}^{D}_h\right)} \xi^2 \llbracket v^\tau_h\rrbracket\cdot \{ \nabla_{h}u^\tau_h \} \ds\dt  + \int_{\mathcal{D}_\tau \times\mathcal{E}^{R}_h} \xi^2 \kappa u^\tau_h v^\tau_h\ds\dt  \nonumber\\ 
& + \int_{\mathcal{D}_\tau \times \left(\mathcal{E}^{i}_h\cup\mathcal{E}^{D}_h\right)} \beta_2\xi^2 \llbracket u^\tau_h\rrbracket \cdot \llbracket v^\tau_h\rrbracket \ds\dt,
\end{align}
\begin{align}\label{lrnndg_l}
l(v^\tau_h)= 
&\int_{\Sigma} f v^\tau_h \dx \dt - \int_{\mathcal{T}_h} u_0 \frac{\partial_\tau v^\tau_h}{\partial_\tau t}(0,\bx)\dx - \int_{\mathcal{T}_h} w_0 v^\tau_h(0,\bx)\dx - \int_{\mathcal{T}_h}\beta_1 u_0 v^\tau_h(0,\bx)\dx \nonumber\\ 
&- \int_{ \mathcal{D}_\tau\times{\cal E}_h^{D}}\xi^2 g_D \boldsymbol{n}\cdot \nabla_{h} v^\tau_{h} \ds\dt + \int_{ \mathcal{D}_\tau\times{\cal E}_h^{N}} \xi^2 g_N v^\tau_h \ds\dt\nonumber\\ 
&+ \int_{ \mathcal{D}_\tau\times{\cal E}_h^{R}}\xi^2 g_R  v^\tau_{h} \ds\dt + \int_{ \mathcal{D}_\tau\times{\cal E}_h^{D}} \beta_2 \xi^2 g_D v^\tau_{h} \ds\dt.
\end{align}

Since we use local RNN for each subdomain, we only need to solve for the parameters $U$ of the output layers of different local networks, while other parameters are randomly assigned and fixed throughout the learning process. From Eq. \eqref{lrnndg_B}, we have the global stiffness matrix $\mathbb{A}$ and from Eq. \eqref{lrnndg_l}, we get the right-hand side $L$. Then, we solve the least-squares problem for the linear system 
$$ \mathbb{A} U= L $$ 
to obtain $U$, which gives us a numerical solution of DVWE.

\begin{lemma}[Consistency]\label{lem:consis}
The LRNN-DG scheme is consistent, i.e., for the solution $u\in C^1(I; H^2(\Omega))$ of the DVWE \eqref{dvwe_strong}, we have
\begin{equation}\label{consis_st}
B_{h\tau}(u,v^\tau_h)= l(v^\tau_h)\quad\forall v^\tau_h\in V^\tau_h.
\end{equation}
\end{lemma}
\begin{proof}
We know that $u\in C^1(I; H^2(\Omega))$ implies $\llbracket u\rrbracket=0$, $[ \nabla u]=0$ in $\mathcal{D}_\tau\times\mathcal{E}^i_h$, $u=g_D$ in $\mathcal{D}_\tau\times\mathcal{E}^D_h$,  $\nabla u \cdot\boldsymbol{n} =g_N$ in $\mathcal{D}_\tau\times\mathcal{E}^N_h$,  $\nabla u \cdot\boldsymbol{n} + \kappa u=g_R$ in $\mathcal{D}_\tau\times\mathcal{E}^R_h$, $[ u(t_i,\bx)] = 0$, $\left[ \frac{\partial u}{\partial t}(t_i,\bx) \right] =0$ in $\mathcal{P}^i_\tau\times\mathcal{T}_h$, $u(t_0,\bx)=u_0(\bx)$ and $\frac {\partial u}{\partial t}(t_0,\bx)=w_0(\bx)$ in $\{t_0\}\times\mathcal{T}_h$. So we can get% Then, by the identities \eqref{ibps} and \eqref{iden_dg}, we have
\begin{equation*}
\begin{aligned}
B_{h\tau}(u,v^\tau_h)=
&- \int_{\Sigma}\frac{\partial u}{\partial t} \frac{\partial_\tau v^\tau_h}{\partial_\tau t} \dx \dt + \int_{\Sigma}\gamma \frac{\partial u}{\partial t} v^\tau_h \dx \dt +\int_{\Sigma}\eta \nabla \left(\frac{\partial u}{\partial t}\right)\cdot\nabla_h v^\tau_h \dx \dt + \int_{\Sigma}\xi^2 \nabla u \cdot\nabla_h v^\tau_h \dx \dt\\ 
&- \int_{\mathcal{T}_h} u_0 \frac{\partial_\tau v^\tau_h}{\partial_\tau t}(0,\bx)\dx +\sum^{N_t}_{i=1}\int_{\mathcal{T}_h}\llbracket v^\tau_h(t_i,\bx) \rrbracket \cdot \left\{\frac{\partial u}{\partial t}(t_i,\bx)\right\}\dx \\
&- \int_{\mathcal{T}_h}\beta_1 u_0 v^\tau_h(0,\bx)\dx -\int_{\mathcal{D}_\tau \times\mathcal{E}^{}_h} \eta \llbracket v^\tau_h\rrbracket\cdot \left\{ \nabla \left(\frac{\partial u}{\partial t}\right) \right\} \ds\dt \\
&- \int_{ \mathcal{D}_\tau\times{\cal E}_h^{D}}\xi^2 g_D \boldsymbol{n}\cdot \nabla_{h} v^\tau_{h} \ds\dt- \int_{\mathcal{D}_\tau \times \left(\mathcal{E}^{i}_h\cup\mathcal{E}^{D}_h\right)} \xi^2 \llbracket v^\tau_h\rrbracket\cdot \{ \nabla u \} \ds\dt \\
&+ \int_{\mathcal{D}_\tau \times\mathcal{E}^{R}_h} \xi^2 (g_R - \nabla u \cdot \boldsymbol{n}) v^\tau_h\ds\dt + \int_{ \mathcal{D}_\tau\times{\cal E}_h^{D}}\xi^2\beta_2 g_D v^\tau_{h} \ds\dt.
\end{aligned}
\end{equation*}
Then, by integration by parts \eqref{ibps}, identities \eqref{iden_dg} and \eqref{iden_tdg}, we have
\begin{equation*}
\begin{aligned}
B_{h\tau}(u,v^\tau_h)=& \int_{\Sigma}\frac{\partial^2 u}{\partial t^2} v^\tau_h \dx \dt  - \int_{\mathcal{T}_h} w_0 v^\tau_h(0,\bx)\dx + \int_{\Sigma}\gamma \frac{\partial u}{\partial t} v^\tau_h \dx \dt\\
&-\int_{\Sigma}\eta \Delta \left(\frac{\partial u}{\partial t}\right) v^\tau_h \dx \dt - \int_{\Sigma}\xi^2 \Delta u v^\tau_h \dx \dt + \int_{ \mathcal{D}_\tau\times{\cal E}_h^{N}} \xi^2 g_N v^\tau_h \ds\dt\\ 
& - \int_{\mathcal{T}_h} u_0 \frac{\partial_\tau v^\tau_h}{\partial_\tau t}(0,\bx)\dx - \int_{\mathcal{T}_h}\beta_1 u_0 v^\tau_h(0,\bx)\dx - \int_{ \mathcal{D}_\tau\times{\cal E}_h^{D}}\xi^2 g_D \boldsymbol{n}\cdot \nabla_{h} v^\tau_{h} \ds\dt \\
& + \int_{\mathcal{D}_\tau \times\mathcal{E}^{R}_h} \xi^2 g_R v^\tau_h\ds\dt + \int_{ \mathcal{D}_\tau\times{\cal E}_h^{D}}\xi^2\beta_2 g_D v^\tau_{h} \ds\dt\\
=& \int_{\Sigma}\left( \frac{\partial^2 u}{\partial t^2} + \gamma \frac{\partial u}{\partial t} - \eta\Delta\left(\frac{\partial u}{\partial t}\right) -\xi^2 \Delta u \right) v^\tau_h \dx \dt  - \int_{\mathcal{T}_h} u_0 \frac{\partial_\tau v^\tau_h}{\partial_\tau t}(0,\bx)\dx \\
& - \int_{\mathcal{T}_h} w_0 v^\tau_h(0,\bx)\dx - \int_{\mathcal{T}_h}\beta_1 u_0 v^\tau_h(0,\bx)\dx - \int_{ \mathcal{D}_\tau\times{\cal E}_h^{D}}\xi^2 g_D \boldsymbol{n}\cdot \nabla_{h} v^\tau_{h} \ds\dt \\
&+ \int_{ \mathcal{D}_\tau\times{\cal E}_h^{N}} \xi^2 g_N v^\tau_h \ds\dt + \int_{\mathcal{D}_\tau \times\mathcal{E}^{R}_h} \xi^2 g_R v^\tau_h\ds\dt + \int_{ \mathcal{D}_\tau\times{\cal E}_h^{D}}\xi^2\beta_2 g_D v^\tau_{h} \ds\dt \\
=& l(v^\tau_h),
\end{aligned}
\end{equation*}
which completes the proof.
\end{proof} 

\subsection{LRNN-$C^0$DG method}\label{lrnnc0dg_dvwe}

In this subsection, we present the LRNN-$C^0$DG method for solving the DVWE. 

In the LRNN-DG method \eqref{timec0for}, the penalty terms enforce and control the Dirichlet boundary conditions, initial conditions, and continuity of numerical solution on interior faces $f$. To remove these penalty terms, we propose the LRNN-$C^0$DG method by imposing that the solution satisfies the Dirichlet boundary conditions, initial conditions, and internal continuity at some collocation points chosen on different faces.

Let us consider the Dirichlet boundary condition Eq. \eqref{strongboundcd}. We select $N_D$ collocation points $P^D_h = \{(t^D,\bx^D) \in \mathcal{D}_\tau\times\mathcal{E}^D_h \}$ and enforce the solution to satisfy the Dirichlet boundary condition on $P^D_h$, that is,
\begin{equation}\label{c0dg_bcd}
u^\tau_h(t^D,\bx^D) = g_D(t^D,\bx^D)\quad \forall(t^D,\bx^D)\in P^D_h.
\end{equation}
Since Neumann and Robin boundary conditions are embedded in the DG scheme, we do not need to choose collocation points on Neumann and Robin boundaries. For the initial condition Eq.  \eqref{stronginit1}, we choose $N_I$ collocation points 
$P^I_\tau = \{(t_0,\bx^I) \in {t_0}\times\mathcal{T}_h \}$, such that
\begin{equation}\label{c0dg_bci}
u^\tau_h(t_0,\bx^I) = u_0(\bx^I)\quad \forall(t_0,\bx^I)\in P^I_\tau.
\end{equation}
Furthermore, we need to ensure that the solution has $C^0$ continuity on interior faces. We take $N_{si}$ collocation points $P^i_{ h} = \{(t^i,\bx^i) \in (\mathcal{D}_\tau\times\mathcal{E}_h^{i}) \}$ and $N_{ti}$ collocation points $P^i_{\tau} = \{(t^i,\bx^i) \in (\mathcal{P}_\tau^{i}\times \mathcal{T}_h) \}$, such that
\begin{equation}\label{c0dg_bcc}
\begin{aligned}
\llbracket u^\tau_h\rrbracket &= 0\quad \forall(t^i,\bx^i)\in P^i_{ h},\\
[ u^\tau_h(t^i,\bx^i)] &= 0\quad \forall(t^i,\bx^i)\in P^i_{\tau}.
\end{aligned}
\end{equation}
We get a system of equations from Eqs. \eqref{c0dg_bcd}, \eqref{c0dg_bci} and \eqref{c0dg_bcc},
\begin{align}\label{c0dg_eqstrong}
\mathbb{A}_2 U=L_2.
\end{align}

The LRNN-$C^0$DG scheme for solving the DVWE \eqref{dvwe_strong} is: Find $u^\tau_h\in V^\tau_h$ such that
\begin{equation}\label{c0for}
\begin{aligned}
B^0_{h\tau}(u^\tau_h,v^\tau_h) = l^0(v^\tau_h) \quad &\forall v^\tau_h\in V^\tau_h,\\
\llbracket u^\tau_h\rrbracket = 0\quad &\forall(t^i,\bx^i)\in P^i_h,\\
[ u^\tau_h(t^i,\bx^i)] = 0\quad &\forall(t^i,\bx^i)\in P^i_\tau,\\
u^\tau_h(t_0,\bx^I) = u_0(\bx^I)\quad &\forall(t_0,\bx^I)\in P^I_\tau,\\
u^\tau_h(t^D,\bx^D) = g_D(\bx^D)\quad &\forall(t^D,\bx^D)\in P^D_h,
\end{aligned}
\end{equation}
where
\begin{align}\label{lrnnc0dg_B}
B^0_{h\tau}(u^\tau_h,v^\tau_h)=
&- \int_{\Sigma}\frac{\partial_\tau u^\tau_h}{\partial_\tau t} \frac{\partial_\tau v^\tau_h}{\partial_\tau t} \dx \dt + \int_{\Sigma}\gamma \frac{\partial_\tau u^\tau_h}{\partial_\tau t} v^\tau_h \dx \dt +\int_{\Sigma}\eta \nabla_h \left(\frac{\partial_\tau u^\tau_h}{\partial_\tau t}\right)\cdot\nabla_h v^\tau_h \dx \dt \nonumber\\ 
& + \int_{\Sigma}\xi^2 \nabla_h u^\tau_h \cdot\nabla_h v^\tau_h \dx \dt + \sum^{N_t}_{i=1}\int_{\mathcal{T}_h}[ v^\tau_h(t_i,\bx) ] \left\{\frac{\partial_\tau u^\tau_h}{\partial_\tau t}(t_i,\bx)\right\}\dx \nonumber\\ 
& -\int_{\mathcal{D}_\tau \times\mathcal{E}^{}_h} \eta \llbracket v^\tau_h\rrbracket\cdot \left\{ \nabla_{h} \left(\frac{\partial_\tau u^\tau_h}{\partial_\tau t}\right) \right\} \ds\dt - \int_{\mathcal{D}_\tau \times \left(\mathcal{E}^{i}_h\cup\mathcal{E}^{D}_h\right)} \xi^2 \llbracket v^\tau_h\rrbracket\cdot \{ \nabla_{h}u^\tau_h \} \ds\dt \nonumber\\  
&+ \int_{\mathcal{D}_\tau \times\mathcal{E}^{R}_h} \xi^2 u^\tau_h v^\tau_h\ds\dt ,
\end{align}
\begin{equation}\label{lrnnc0dg_l}
l^0(v^\tau_h)= 
\int_{\Sigma} f v^\tau_h \dx \dt  - \int_{\mathcal{T}_h} w_0 v^\tau_h(0,\bx)\dx
+ \int_{ \mathcal{D}_\tau\times{\cal E}_h^{N}} \xi^2 g_N v^\tau_h \ds\dt
+ \int_{ \mathcal{D}_\tau\times{\cal E}_h^{R}}\xi^2 g_R  v^\tau_{h} \ds\dt .
\end{equation}
From Eqs. \eqref{lrnnc0dg_B} and \eqref{lrnnc0dg_l}, we get the global stiffness matrix $\mathbb{A}_1$ and the right-hand side $L_1$ of the LRNN-$C^0$DG method. Then, we use the least-squares method to solve the system of equations
\begin{equation}
\begin{bmatrix} \mathbb{A}_1 \\ \mathbb{A}_2 \end{bmatrix} U_0 = \begin{bmatrix} L_1 \\ L_2 \end{bmatrix},
\end{equation}
where $U_0$ is the unknown parameters of the output layers in the LRNN-$C^0$DG method.

\begin{remark}
In \cite{Sun2022lrnndg}, a numerical experiment demonstrates that the system of equations \eqref{c0dg_eqstrong} ensures that the solution satisfies Dirichlet boundary conditions, initial condition, and interior continuity. Specifically, the jump $\llbracket u^\tau_h\rrbracket \approx 0$ on interior faces, $u^\tau_h \approx g_D$ on Dirichlet edges, and $u^\tau_h \approx u_0$ on initial edges in the $L^2$ norm when there are sufficient collocation points. 
\end{remark}

\subsection{LRNN-$C^1$DG method}\label{lrnnc1dg_dvwe}

In this subsection, we introduce the LRNN-$C^1$DG method for solving the DVWE. 

We multiply both sides of Eq. \eqref{strongeq} by test functions $v^\tau_h\in V^\tau_h$ and integrate over subdomain $\sigma = I_i\times K$. Then we apply integration by parts and get the following weak formulation:
\begin{align}\label{lrnn_c1dg_weak}
B^1_{h\tau}(u^\tau_h,v^\tau_h) = \int_{\sigma} f v^\tau_h \dx \dt \quad \forall v^\tau_h\in V^\tau_h,\quad\forall \sigma \in \mathcal{D}_\tau\times\mathcal{T}_h,
\end{align}
where
\begin{align*}
B^1_{h\tau}(u^\tau_h,v^\tau_h)=
&- \int_{\sigma}\frac{\partial u^\tau_h}{\partial t} \frac{\partial v^\tau_h}{\partial t} \dx \dt + \int_{\sigma}\gamma \frac{\partial u^\tau_h}{\partial t} v^\tau_h \dx \dt +\int_{\sigma}\eta \nabla \left(\frac{\partial u^\tau_h}{\partial t}\right)\cdot\nabla v^\tau_h \dx \dt\\ 
&+ \int_{\sigma}\xi^2 \nabla u^\tau_h \cdot\nabla v^\tau_h \dx \dt + \int_{{K}}\frac{\partial u^\tau_h}{\partial t}(t,\bx)v^\tau_h(t,\bx)|^{t_i}_{t_{i-1}}\dx\\
&-\int_{I_i\times\partial K}\eta \nabla \left(\frac{\partial u^\tau_h}{\partial t}\right)\cdot\boldsymbol{n}_K v^\tau_h \ds \dt -\int_{I_i\times\partial K}\xi^2 \nabla u^\tau_h \cdot\boldsymbol{n}_K v^\tau_h \ds \dt.
\end{align*}
The initial condition \eqref{stronginit2}, Neumann boundary conditions \eqref{strongboundcn}, and Robin boundary conditions \eqref{strongboundcr} are incorporated into the weak formulation \eqref{lrnn_c1dg_weak} naturally.

Similar to the LRNN-$C^0$DG method, we need to impose some conditions to ensure the continuity of numerical solution across faces (edges) between subdomains. Moreover, we choose $N_{si}$ collocation points $P^i_{ h} = \{(t^i,\bx^i) \in \mathcal{D}_\tau\times\mathcal{E}_h^{i} \}$ and $N_{ti}$ collocation points $P^i_{\tau} = \{(t^i,\bx^i) \in \mathcal{P}_\tau^{i}\times \mathcal{T}_h \}$ on interior faces and require that the solution has $C^1$ continuity on these faces, 
\begin{subequations}\label{c1dg_bcc}
\begin{align}
\llbracket u^\tau_h\rrbracket = 0\quad &\forall(t^i,\bx^i)\in P^i_{ h},\label{c1dg_sc0}\\
[ u^\tau_h(t^i,\bx^i)] = 0\quad &\forall(t^i,\bx^i)\in P^i_{\tau},\label{c1dg_tc0}\\
\left[ \nabla_hu^\tau_h\right] = 0\quad &\forall(t^i,\bx^i)\in P^i_{ h},\label{c1dg_hc1}\\
\left[ \frac{\partial_\tau u^\tau_h}{\partial_\tau t}(t^i,\bx^i)\right] = 0\quad &\forall(t^i,\bx^i)\in P^i_{\tau}.\label{c1dg_tc1}
\end{align}
\end{subequations}

Considering the initial condition \eqref{stronginit1} and Dirichlet boundary condition \eqref{strongboundcd}, we select $N_I$ collocation points $P^I_\tau = \{(t_0,\bx^I) \in {t_0}\times\mathcal{T}_h \}$ and $N_D$ collocation points $P^D_h = \{(t^D,\bx^D) \in \mathcal{D}_\tau\times\mathcal{E}^D_h \}$ such that the solution $u^\tau_h$ satisfies
\begin{equation}\label{c1dg_bci}
u^\tau_h(t_0,\bx^I) = u_0(\bx^I)\quad \forall(t_0,\bx^I)\in P^I_\tau,
\end{equation}
\begin{equation}\label{c1dg_bcd}
u^\tau_h(t^D,\bx^D) = g_D(t^D,\bx^D)\quad \forall(t^D,\bx^D)\in P^D_h.
\end{equation}
Then the LRNN-$C^1$DG method for solving DVWE is: Find $u^\tau_h \in V^\tau_h$ that satisfies Eqs. \eqref{lrnn_c1dg_weak}--\eqref{c1dg_bcd}.

From Eq. \eqref{lrnn_c1dg_weak}, we obtain a stiffness matrix $\mathbb{A}_3$ and a right-hand side $L_3$. From Eqs. \eqref{c1dg_bcc}--\eqref{c1dg_bcd}, we obtain another matrix $\mathbb{A}_4$ and a right-hand side $L_4$. Finally, we solve the following system of equations by least-squares methods
\begin{equation}
\begin{bmatrix} \mathbb{A}_3 \\ \mathbb{A}_4 \end{bmatrix} U_1 = \begin{bmatrix} L_3 \\ L_4 \end{bmatrix},
\end{equation}
where $U_1$ is the unknown parameters of output layers of local RNNs.

\begin{remark}
LRNN-DG methods use a space-time approach that treats the temporal and spatial variables in a consistent manner. This has some advantages, such as high stability and avoiding the error accumulation caused by time iterations. 
\end{remark}

\section{Numerical examples}
\label{numericalex}

In this section, we present some numerical experiments to demonstrate the performance of the proposed methods for DVWE. 

We use the Pytorch library in Python to construct local neural networks on different subdomains. Each local neural network has one hidden layer. The parameters of the hidden layer are randomly generated from the uniform distribution $U(-r,r)$ and fixed throughout the training process, where $r$ is a positive constant. The influence of the parameter $r$ is discussed in \cite{Dong2021locELMW}. We apply the Gauss quadrature formula to evaluate all the integrals in the proposed methods. For solving the parameters of the output layers, we use the least-squares method implemented by the wrapper function in the Scipy package. We denote by ${\rm DoF}_\sigma$ the number of degrees of freedom in each subdomain $\sigma$. We define the global relative $L^2$ error as 
$$E^{L^2}_r (u) = \left(\int_\Sigma (u-u^*)^2 \dx\dt\right)^{\frac{1}{2}}\Big / \left(\int_\Sigma (u^*)^2 \dx\dt\right)^{\frac{1}{2}},$$ 
and the global relative $H^1$ error as 
$$E^{H^1}_r(u) = \left(\int_\Sigma (u_t-u_t^*)^2 + |\nabla (u - u^*)|^2 \dx\dt\right)^{\frac{1}{2}}
\Big / \left(\int_\Sigma (u_t^*)^2 + |\nabla u^*|^2 \dx\dt\right)^{\frac{1}{2}},$$ 
where $u$ is the numerical solution and $u^*$ is the real solution.

\begin{example}[1-D case]
\label{1ddvwe}
We consider a one-dimensional diffusive-viscous wave equation on the domain $\Omega = (0,1)$ and the time interval $I=(0,1)$. 
The Dirichlet boundary is $\Gamma_D = \partial\Omega$. We choose the coefficients for water-saturated rock, i.e., $\gamma=90$, $\eta=2\times 10^{-7}$, and $\xi=1.47$. The exact solution is
\begin{equation}
u(t,x) = e^{cos(27x^2 + 27t^2 + 8\pi - 24)}\frac{(x^2 + 1)}{2}\frac{(t^2 + 1)}{2},\nonumber
\end{equation}
which determines the Dirichlet boundary condition $g_D$, the initial conditions $u_0$ and $w_0$, and the source term $f$.
\end{example}

\begin{table}[]
\centering
\begin{tabular}{|c|cc|cc|cc|}
\hline
{$\tau$}, {$h$}             & \multicolumn{2}{c|}{$1/6$, $1/6$}                                                    & \multicolumn{2}{c|}{$1/8$, $1/8$}                                                    & \multicolumn{2}{c|}{$1/10$, $1/10$}                                                   \\ \hline
\diagbox[width=5em,trim=l]{${\rm DoF}_\sigma$}{Norm} & \multicolumn{1}{c|}{ $L^{2}$} &  $H^{1}$ & \multicolumn{1}{c|}{ $L^{2}$} & \multicolumn{1}{c|}{ $H^{1}$} & \multicolumn{1}{c|}{ $L^{2}$} & \multicolumn{1}{c|}{ $H^{1}$} \\ \hline
80                         & \multicolumn{1}{l|}{7.85E-01}         & 5.10E+00                              & \multicolumn{1}{l|}{6.48E-02}         & 4.51E-01                              & \multicolumn{1}{l|}{8.52E-03}         & 7.80E-02                              \\ \hline
160                        & \multicolumn{1}{l|}{1.14E-02}         & 7.81E-02                              & \multicolumn{1}{l|}{2.36E-03}         & 2.17E-02                              & \multicolumn{1}{l|}{3.08E-04}         & 3.88E-03                              \\ \hline
320                        & \multicolumn{1}{l|}{4.40E-03}         & 3.45E-02                              & \multicolumn{1}{l|}{5.00E-04}         & 5.53E-03                              & \multicolumn{1}{l|}{6.57E-05}         & 9.19E-04                              \\ \hline
\end{tabular}
\caption{Global relative errors of the space-time LRNN-DG method in Example \ref{1ddvwe}}
\label{table1dlrnndgerr}
\end{table}
\begin{table}[]
\centering
\begin{tabular}{|c|cc|cc|cc|}
\hline
{$\tau$}, {$h$}             & \multicolumn{2}{c|}{$1/6$, $1/6$}                                                    & \multicolumn{2}{c|}{$1/8$, $1/8$}                                                    & \multicolumn{2}{c|}{$1/10$, $1/10$}                                                   \\ \hline
\diagbox[width=5em,trim=l]{${\rm DoF}_\sigma$}{Norm} & \multicolumn{1}{c|}{ $L^{2}$} &  $H^{1}$ & \multicolumn{1}{c|}{ $L^{2}$} & \multicolumn{1}{c|}{ $H^{1}$} & \multicolumn{1}{c|}{ $L^{2}$} & \multicolumn{1}{c|}{ $H^{1}$} \\ \hline
80                         & \multicolumn{1}{l|}{2.10E-01}         & 9.72E-01                              & \multicolumn{1}{l|}{5.30E-02}         & 3.21E-01                              & \multicolumn{1}{l|}{1.29E-02}         & 9.55E-02                              \\ \hline
160                        & \multicolumn{1}{l|}{1.29E-02}         & 8.31E-02                              & \multicolumn{1}{l|}{2.07E-03}         & 1.89E-02                              & \multicolumn{1}{l|}{2.58E-04}         & 2.79E-03                              \\ \hline
320                        & \multicolumn{1}{l|}{5.19E-03}         & 3.50E-02                              & \multicolumn{1}{l|}{4.88E-04}         & 5.10E-03                              & \multicolumn{1}{l|}{8.27E-05}         & 9.50E-04                              \\ \hline
\end{tabular}
\caption{Global relative errors of the space-time LRNN-$C^0$DG method in Example \ref{1ddvwe}}
\label{table1dlrnnc0dgerr}
\end{table}
\begin{table}[]
\centering
\begin{tabular}{|c|cc|cc|cc|}
\hline
{$\tau$}, {$h$}             & \multicolumn{2}{c|}{$1/6$, $1/6$}                                                    & \multicolumn{2}{c|}{$1/8$, $1/8$}                                                    & \multicolumn{2}{c|}{$1/10$, $1/10$}                                                   \\ \hline
\diagbox[width=5em,trim=l]{${\rm DoF}_\sigma$}{Norm} & \multicolumn{1}{c|}{ $L^{2}$} &  $H^{1}$ & \multicolumn{1}{c|}{ $L^{2}$} & \multicolumn{1}{c|}{ $H^{1}$} & \multicolumn{1}{c|}{ $L^{2}$} & \multicolumn{1}{c|}{ $H^{1}$} \\ \hline
80                         & \multicolumn{1}{l|}{1.30E-01}         & 5.17E-01                              & \multicolumn{1}{l|}{2.85E-02}         & 1.51E-01                              & \multicolumn{1}{l|}{1.29E-02}         & 8.03E-02                              \\ \hline
160                        & \multicolumn{1}{l|}{9.29E-03}         & 6.23E-02                              & \multicolumn{1}{l|}{2.95E-03}         & 2.28E-02                              & \multicolumn{1}{l|}{4.86E-04}         & 4.79E-03                              \\ \hline
320                        & \multicolumn{1}{l|}{4.11E-03}         & 2.79E-02                              & \multicolumn{1}{l|}{3.62E-04}         & 3.65E-03                              & \multicolumn{1}{l|}{7.53E-05}         & 8.00E-04                              \\ \hline
\end{tabular}
\caption{Global relative errors of the space-time LRNN-$C^1$DG method in Example \ref{1ddvwe}}
\label{table1dlrnnc1dgerr}
\end{table}

\begin{figure}[htb] 
  \centering  
   
  \begin{subfigure}{0.4\textwidth}
      \centering      
      \includegraphics[width=0.9\textwidth]{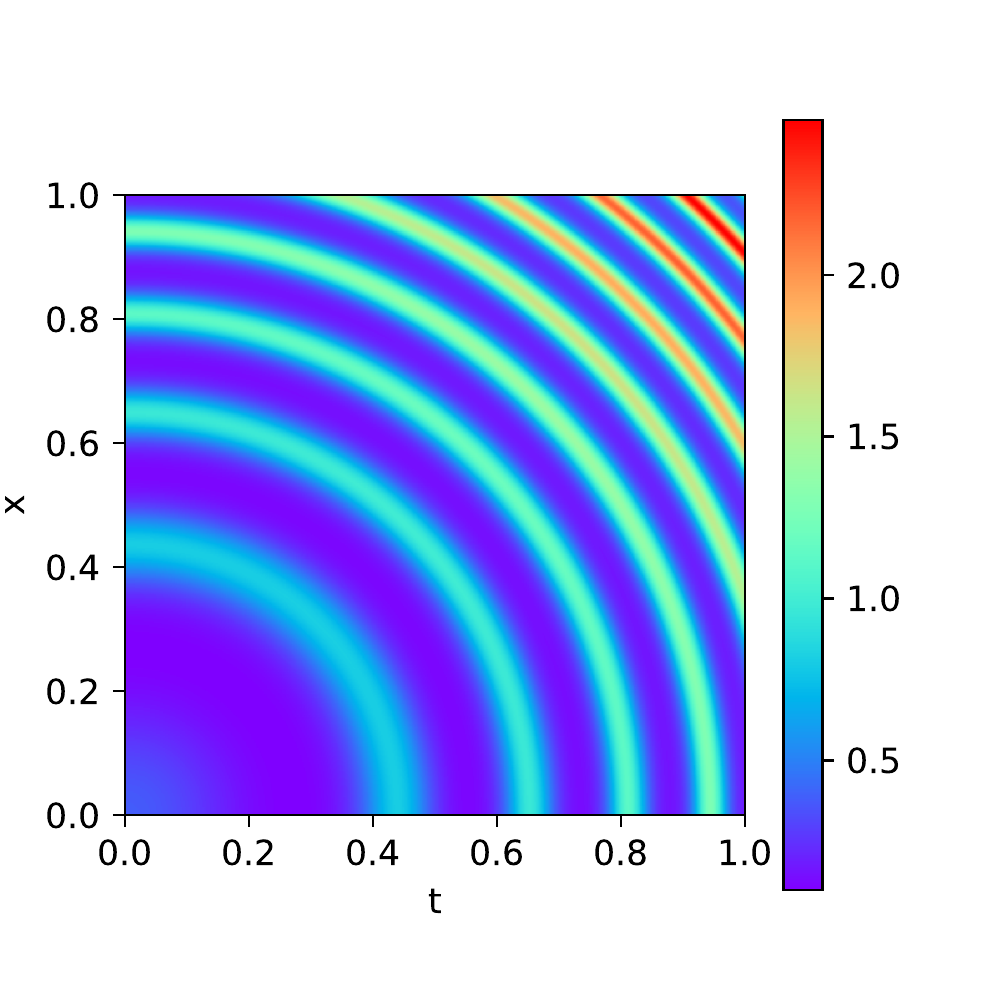}  
      \caption{Exact solution}
      \label{fig:a}
  \end{subfigure}\quad
  \begin{subfigure}{0.4\textwidth}
      \centering      
      \includegraphics[width=0.9\textwidth]{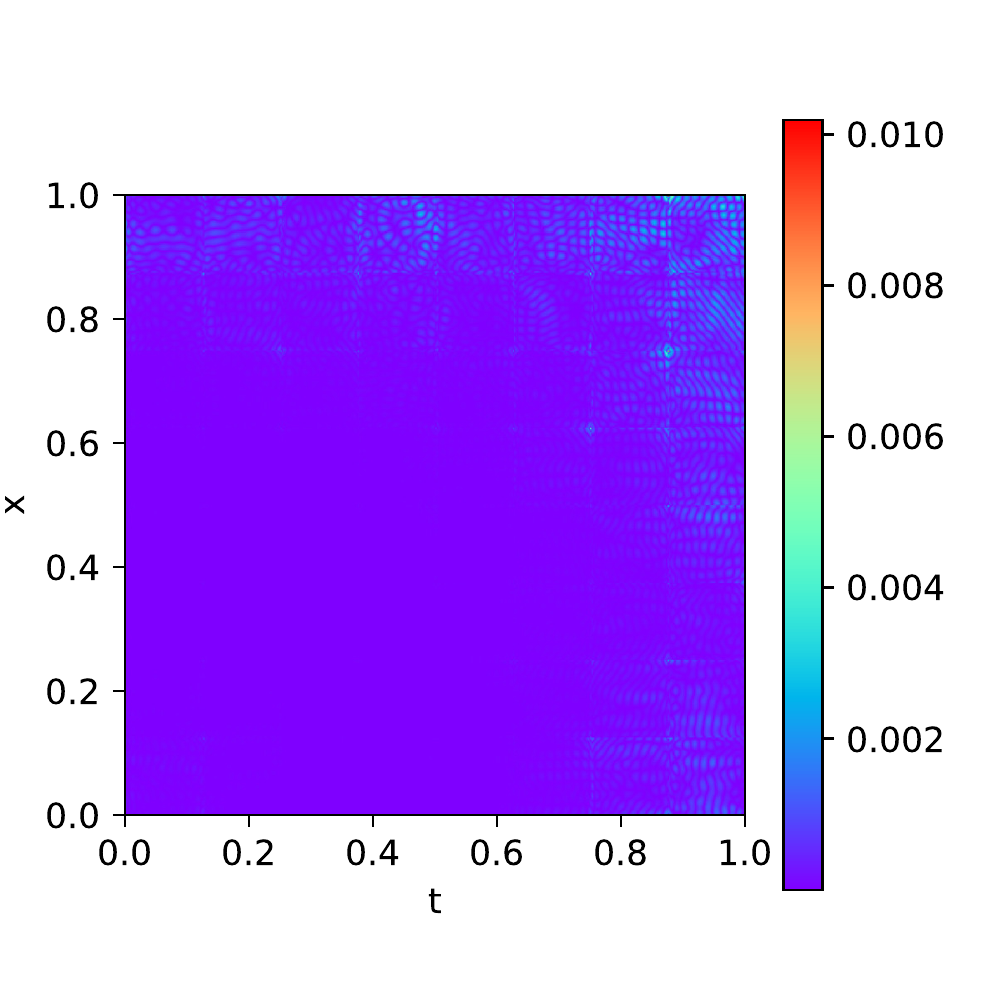}  
      \caption{Absolute errors of LRNN-DG method}
      \label{fig:b}
  \end{subfigure}\\
  \begin{subfigure}{0.4\textwidth}
      \centering    
      \includegraphics[width=0.9\textwidth]{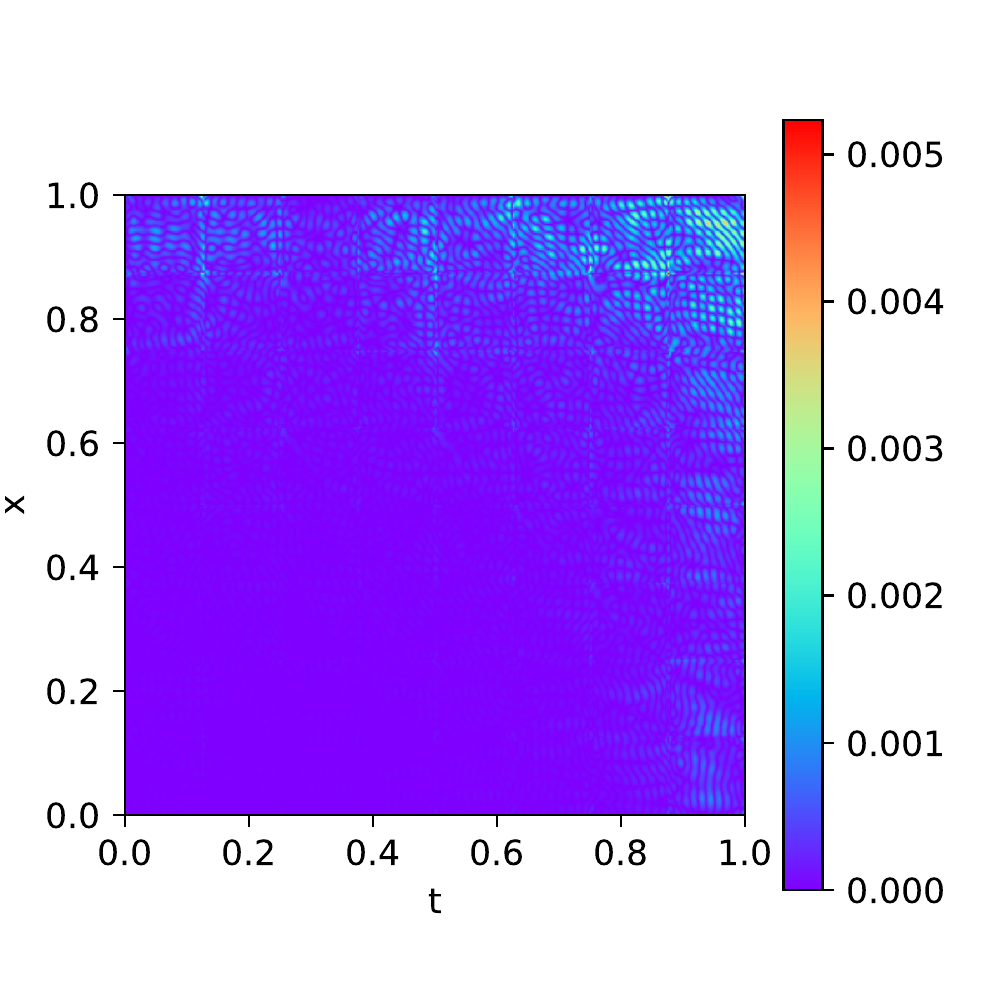}  
		 \caption{Absolute errors of LRNN-$C^0$DG method}
      \label{fig:c}
    \end{subfigure}  \quad
   \begin{subfigure}{0.4\textwidth}
      \centering    
      \includegraphics[width=0.9\textwidth]{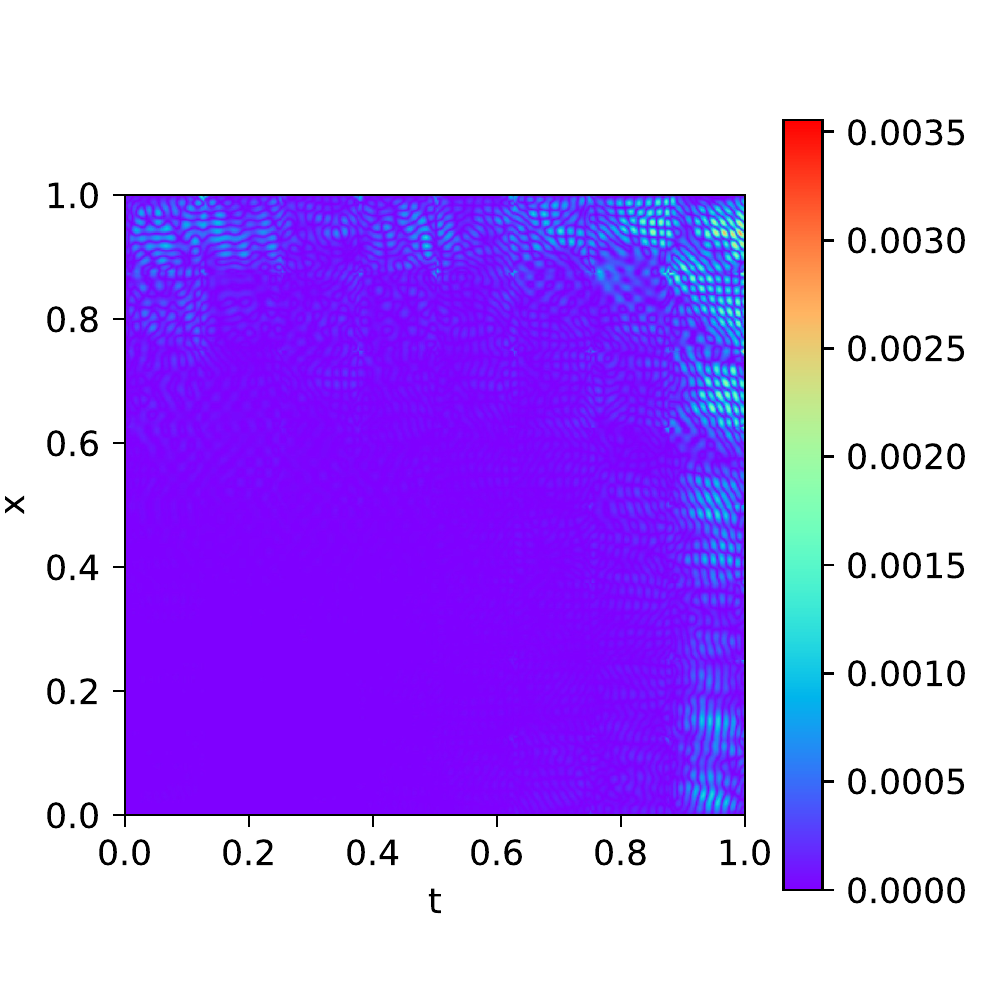}  
		 \caption{Absolute errors of LRNN-$C^1$DG method}
      \label{fig:d}
    \end{subfigure}  
  \caption{The exact solution and absolute errors computed by proposed methods in Example \ref{1ddvwe}.}  
  \label{figure2ddgerr}
\end{figure}

The global relative $L^2$ errors and $H^1$ errors of the LRNN-DG, LRNN-$C^0$DG, and LRNN-$C^1$DG methods are shown in Tables \ref{table1dlrnndgerr}, \ref{table1dlrnnc0dgerr} and \ref{table1dlrnnc1dgerr} for different values of the mesh size $h$, the time sub-interval $\tau$ and the number of degrees of freedom ${\rm DoF}_\sigma$ in each domain. 
We use 15 quadrature points in each dimension for all three methods.
For the LRNN-DG method, we set $r = 1.38$, $\beta_1 = \beta_2 = 7$. For the LRNN-$C^0$DG and LRNN-$C^1$DG methods, we choose $r = 1.34$, 13 collocation points on every edge. We find that the numerical solutions become more accurate as $h$ and $\tau$ decrease and ${\rm DoF}\sigma$ increases. The three methods are all accurate and have similar performance.

The oscillatory exact solution in Example \ref{1ddvwe} and the absolute errors of the proposed methods are displayed in Figure \ref{figure2ddgerr}. We use the same settings as in the previous tables, with the mesh size $h = 1/8$ and the number of degrees of freedom ${\rm DoF}_\sigma=320$. The proposed methods perform well in this example, and they do not need time discretization, so they avoid the error accumulation over time that happens in the traditional finite difference methods.

\begin{example}[2-D case]
\label{2ddvwe}

We consider a two-dimensional diffusive-viscous wave equation on the domain $\Omega =(0,1)^2$ and the time interval $I = (0,0.5)$ with Dirichlet boundary conditions. The parameters $\gamma=1$, $\eta=0.01$ and $\xi=0.1$, which are consistent with the numerical test in \cite{Ling2023LDG}. The exact solution is
\begin{equation}
u(t,x,y) = e^{-t}sin(2\pi x)sin(2\pi y).\nonumber
\end{equation}
The Dirichlet boundary condition $g_D$, the initial conditions $u_0$ and $w_0$ and the source term $f$ can be derived from the exact solution.
\end{example}

\begin{table}[]
\centering
\begin{tabular}{|c|cc|cc|cc|}
\hline
{$\tau$}, {$h$}             & \multicolumn{2}{c|}{$1/4$, $1/2$}                         & \multicolumn{2}{c|}{$1/6$, $1/3$}                                              & \multicolumn{2}{c|}{$1/8$, $1/4$}                                              \\ \hline
\diagbox[width=5em,trim=l]{${\rm DoF}_\sigma$}{Norm} & \multicolumn{1}{c|}{ $L^{2}$} &  $H^{1}$ & \multicolumn{1}{c|}{ $L^{2}$} & \multicolumn{1}{c|}{ $H^{1}$} & \multicolumn{1}{c|}{ $L^{2}$} & \multicolumn{1}{c|}{ $H^{1}$} \\ \hline
40                         & \multicolumn{1}{c|}{1.75E+00}         & 3.00E+00         & \multicolumn{1}{l|}{6.61E-01}         & 1.38E+00                              & \multicolumn{1}{l|}{3.36E+00}         & 9.64E+00                              \\ \hline 
80                         & \multicolumn{1}{c|}{4.46E-01}         & 1.06E+00         & \multicolumn{1}{l|}{5.66E-01}         & 1.84E+00                              & \multicolumn{1}{l|}{1.16E-01}         & 4.93E-01                              \\ \hline
160                        & \multicolumn{1}{c|}{1.17E-02}         & 3.73E-02         & \multicolumn{1}{l|}{5.48E-03}         & 2.49E-02                              & \multicolumn{1}{l|}{3.51E-03}         & 1.89E-02                              \\ \hline
320                        & \multicolumn{1}{c|}{2.29E-04}         & 1.09E-03         & \multicolumn{1}{l|}{2.73E-05}         & 1.95E-04                              & \multicolumn{1}{l|}{1.58E-05}         & 1.40E-04                              \\ \hline
640                        & \multicolumn{1}{c|}{3.77E-05}         & 2.11E-04         & \multicolumn{1}{l|}{7.15E-06}         & 6.50E-05                              & \multicolumn{1}{l|}{2.94E-06}         & 3.50E-05                              \\ \hline
\end{tabular}
\caption{Global relative errors of the space-time LRNN-DG method in Example \ref{2ddvwe}}
\label{table2dlrnndgerr}
\end{table}

\begin{table}[]
\centering
\begin{tabular}{|c|cc|cc|cc|}
\hline
{$\tau$}, {$h$}             & \multicolumn{2}{c|}{$1/4$, $1/2$}                         & \multicolumn{2}{c|}{$1/6$, $1/3$}                         & \multicolumn{2}{c|}{$1/8$, $1/4$}                         \\ \hline
\diagbox[width=5em,trim=l]{${\rm DoF}_\sigma$}{Norm} & \multicolumn{1}{c|}{ $L^{2}$} &  $H^{1}$ & \multicolumn{1}{c|}{ $L^{2}$} & \multicolumn{1}{c|}{ $H^{1}$} & \multicolumn{1}{c|}{ $L^{2}$} & \multicolumn{1}{c|}{ $H^{1}$} \\ \hline
40                         & \multicolumn{1}{c|}{3.00E-01}         & 5.62E-01         & \multicolumn{1}{c|}{5.50E-02}         & 1.36E-01         & \multicolumn{1}{c|}{6.86E-02}         & 1.89E-01         \\ \hline
80                         & \multicolumn{1}{c|}{7.13E-02}         & 1.64E-01         & \multicolumn{1}{c|}{3.03E-02}         & 1.02E-01         & \multicolumn{1}{c|}{8.50E-03}         & 3.49E-02         \\ \hline
160                        & \multicolumn{1}{c|}{9.45E-03}         & 3.17E-02         & \multicolumn{1}{c|}{1.85E-03}         & 8.26E-03         & \multicolumn{1}{c|}{8.51E-04}         & 4.54E-03         \\ \hline
320                        & \multicolumn{1}{c|}{5.80E-05}         & 2.83E-04         & \multicolumn{1}{c|}{3.05E-05}         & 1.58E-04         & \multicolumn{1}{c|}{5.76E-06}         & 4.73E-05         \\ \hline
640                        & \multicolumn{1}{c|}{2.18E-05}         & 1.07E-04         & \multicolumn{1}{c|}{7.24E-06}         & 5.12E-05         & \multicolumn{1}{c|}{2.99E-06}         & 2.56E-05         \\ \hline
\end{tabular}
\caption{Global relative errors of the space-time LRNN-$C^0$DG method in Example \ref{2ddvwe}}
\label{table2dlrnnc0dgerr}
\end{table}

\begin{table}[]
\centering
\begin{tabular}{|c|cc|cc|cc|}
\hline
{$\tau$}, {$h$}             & \multicolumn{2}{c|}{$1/4$, $1/2$}                         & \multicolumn{2}{c|}{$1/6$, $1/3$}                         & \multicolumn{2}{c|}{$1/8$, $1/4$}                         \\ \hline
\diagbox[width=5em,trim=l]{${\rm DoF}_\sigma$}{Norm} & \multicolumn{1}{c|}{ $L^{2}$} &  $H^{1}$ & \multicolumn{1}{c|}{ $L^{2}$} & \multicolumn{1}{c|}{ $H^{1}$} & \multicolumn{1}{c|}{ $L^{2}$} & \multicolumn{1}{c|}{ $H^{1}$} \\ \hline
40                         & \multicolumn{1}{c|}{4.51E-01}         & 5.36E-01         & \multicolumn{1}{c|}{1.69E-01}         & 2.05E-01         & \multicolumn{1}{c|}{2.00E-01}         & 2.50E-01         \\ \hline
80                         & \multicolumn{1}{c|}{3.81E-02}         & 7.74E-02         & \multicolumn{1}{c|}{2.97E-02}         & 6.39E-02         & \multicolumn{1}{c|}{1.22E-02}         & 3.01E-02         \\ \hline
160                        & \multicolumn{1}{c|}{1.23E-02}         & 2.87E-02         & \multicolumn{1}{c|}{1.33E-03}         & 4.34E-03         & \multicolumn{1}{c|}{7.08E-04}         & 2.73E-03         \\ \hline
320                        & \multicolumn{1}{c|}{6.82E-05}         & 1.84E-04         & \multicolumn{1}{c|}{9.92E-06}         & 3.52E-05         & \multicolumn{1}{c|}{3.91E-06}         & 1.53E-05         \\ \hline
640                        & \multicolumn{1}{c|}{4.11E-05}         & 1.08E-04         & \multicolumn{1}{c|}{5.33E-06}         & 2.02E-05         & \multicolumn{1}{c|}{2.81E-06}         & 1.03E-05         \\ \hline
\end{tabular}
\caption{Global relative errors of the space-time LRNN-$C^1$DG method in Example \ref{2ddvwe}}
\label{table2dlrnnc1dgerr}
\end{table}

\begin{table}[]
\centering
\begin{tabular}{|c|cc|cc|cc|}
\hline
{$\tau$}, {$h$}             & \multicolumn{2}{c|}{$1/4$, $1/4$}                         & \multicolumn{2}{c|}{$1/4$, $1/5$}                         & \multicolumn{2}{c|}{$1/4$, $1/6$}                         \\ \hline
\diagbox[width=5em,trim=l]{${\rm DoF}_\sigma$}{Norm} & \multicolumn{1}{c|}{ $L^{2}$} &  $H^{1}$ & \multicolumn{1}{c|}{ $L^{2}$} & \multicolumn{1}{c|}{ $H^{1}$} & \multicolumn{1}{c|}{ $L^{2}$} & \multicolumn{1}{c|}{ $H^{1}$} \\ \hline
40                         & \multicolumn{1}{c|}{1.64E+00}         & 5.36E+00         & \multicolumn{1}{c|}{1.40E-01}         & 6.13E-01         & \multicolumn{1}{c|}{2.10E-01}         & 9.24E-01         \\ \hline
80                         & \multicolumn{1}{c|}{5.59E-02}         & 2.59E-01         & \multicolumn{1}{c|}{5.29E-01}         & 2.66E+00         & \multicolumn{1}{c|}{5.90E-02}         & 3.86E-01         \\ \hline
160                        & \multicolumn{1}{c|}{2.42E-03}         & 1.40E-02         & \multicolumn{1}{c|}{1.50E-03}         & 1.18E-02         & \multicolumn{1}{c|}{8.38E-04}         & 7.27E-03         \\ \hline
320                        & \multicolumn{1}{c|}{1.20E-05}         & 1.04E-04         & \multicolumn{1}{c|}{5.15E-06}         & 5.45E-05         & \multicolumn{1}{c|}{3.17E-06}         & 4.40E-05         \\ \hline
640                        & \multicolumn{1}{c|}{2.51E-06}         & 2.85E-05         & \multicolumn{1}{c|}{1.33E-06}         & 1.98E-05         & \multicolumn{1}{c|}{8.21E-07}         & 1.61E-05         \\ \hline
\end{tabular}
\caption{Global relative errors of the space-time LRNN-DG method with $\tau = 1/4$ in Example \ref{2ddvwe}}
\label{table2dlrnndgerr_ts}
\end{table}

Table \ref{table2dlrnndgerr}, Table \ref{table2dlrnnc0dgerr} and Table \ref{table2dlrnnc1dgerr} show the global relative $L^2$ errors and $H^1$ errors of the proposed methods for this problem. The settings of the proposed methods are as follows. We use 9 quadrature points in each dimension for all three methods. For the LRNN-DG method, we set $r = 0.6$ and penalty parameters $\beta_1 = \beta_2 = 5$. For the LRNN-$C^0$DG method, we use $r = 0.57$ and 39 collocation points on each face. For the LRNN-$C^1$DG method, we use $r = 0.44$ and 32 collocation points on each face. We investigate the relation of the mesh size $h$, the time sub-interval $\tau$, and the number of degrees of freedom ${\rm DoF}_\sigma$ to the accuracy. The errors of the proposed methods decrease as $h$ and $\tau$ decrease and ${\rm DoF}_\sigma$ increases. The three methods perform similarly.

Furthermore, the time interval of this problem is short, so we conduct another set of experiments with a fixed time sub-interval $\tau = 1/4$ and a smaller mesh size. We demonstrate that this setting allows the proposed method to achieve similar accuracies with fewer total degrees of freedom. Table \ref{table2dlrnndgerr_ts} shows the performance of the LRNN-DG method in this setting and the other parameters are the same as in the previous experiments. The other two methods have similar results that are omitted here.

We also calculate the traditional $L^2$ errors and $H^1$ errors at $t= 0.5$ to compare with the local DG method proposed in \cite{Ling2023LDG}. We find that the proposed methods achieve competitive accuracy with the local DG method in \cite{Ling2023LDG} when the total degrees of freedom of the proposed methods are similar to the number of spatial degrees of freedom of the local DG method. However, note that the space-time LRNN-DG methods only require solving a linear least square problem, while the local DG method involves many time iterations.

\begin{example}[3-D case]
\label{3ddvwe}
We present a numerical experiment to solve the diffusive-viscous wave equation in 3-D space $\Omega =(0,1)^3$ and time interval $I = (0,5)$ with mixed boundary conditions,
where $\Gamma_D = \left(\{0,1\}\times[0,1]\times[0,1]\right)\cup \left([0,1]\times\{0,1\}\times[0,1]\right)$, $\Gamma_N = \left((0,1)\times(0,1)\times\{0\}\right)$, $\Gamma_R = \left((0,1)\times(0,1)\times\{1\}\right)$. We use $\kappa=1$, and parameters $\gamma=56$, $\eta=5.6\times 10^{-8}$ and $\xi=1.19$ for dry sandstone. The exact solution is
\begin{equation}
u(t,x,y,z) = t^2 sin(\pi x)sin(\pi y)sin(\pi z).\nonumber
\end{equation}
The prescribed functions $g_D$, $g_N$, and $g_R$, the initial conditions $u_0$, $w_0$, and the source term $f$ can be derived from the exact solution.
\end{example}

\begin{table}[]
\centering
\begin{tabular}{|c|ccc|ccc|}
\hline
{$\tau$}, {$h$}             & \multicolumn{3}{c|}{$5/2$, $1/2$}                                                                                                    & \multicolumn{3}{c|}{$5/3$, $1/3$}                                                                        \\ \hline
\diagbox[width=5em,trim=l]{${\rm DoF}_\sigma$}{Norm} & \multicolumn{1}{c|}{ $L^{2}$ }           & \multicolumn{1}{c|}{ $H^{1}$ } & time                   & \multicolumn{1}{c|}{ $L^{2}$ } & \multicolumn{1}{c|}{ $H^{1}$ } & time \\ \hline
40                         & \multicolumn{1}{c|}{9.54E-02}                        & \multicolumn{1}{c|}{2.15E-01}              & 9.46E-01                        & \multicolumn{1}{c|}{5.35E-02}              & \multicolumn{1}{c|}{1.48E-01}              & 3.74E+00      \\ \hline
80                         & \multicolumn{1}{c|}{2.53E-02}                        & \multicolumn{1}{c|}{6.37E-02}              & 2.28E+00
& \multicolumn{1}{c|}{7.44E-02}              & \multicolumn{1}{c|}{1.97E-01}              & 2.85E+01      \\ \hline
160                        & \multicolumn{1}{c|}{9.28E-03}                        & \multicolumn{1}{c|}{3.25E-02}              & 8.31E+00                        & \multicolumn{1}{c|}{4.81E-03}              & \multicolumn{1}{c|}{3.52E-02}              & 1.98E+02      \\ \hline
320                        & \multicolumn{1}{c|}{2.22E-04} & \multicolumn{1}{c|}{1.19E-03}              & {3.35E+01} & \multicolumn{1}{c|}{3.56E-05}              & \multicolumn{1}{c|}{2.89E-04}              & 1.13E+03      \\ \hline
\end{tabular}
\caption{Relative errors of the space-time LRNN-DG method when $t=5$ and total training time in Example \ref{3ddvwe}}
\label{table3dlrnndgerr}
\end{table}

\begin{table}[]
\centering
\begin{tabular}{|c|cc|cc|}
\hline
{$\tau$}, {$h$}             & \multicolumn{2}{c|}{$5/2$, $1/2$}                                   & \multicolumn{2}{c|}{$5/3$, $1/3$}                                   \\ \hline
\diagbox[width=5em,trim=l]{${\rm DoF}_\sigma$}{Norm} & \multicolumn{1}{c|}{$L^{2}$}           & $H^{1}$ & \multicolumn{1}{c|}{$L^{2}$} & $H^{1}$ \\ \hline
40                         & \multicolumn{1}{c|}{1.35E-01}              & 3.33E-01              & \multicolumn{1}{c|}{5.02E-02}              & 1.96E-01              \\ \hline
80                         & \multicolumn{1}{c|}{1.10E-01}              & 3.08E-01              & \multicolumn{1}{c|}{3.80E-02}              & 1.54E-01              \\ \hline
160                        & \multicolumn{1}{c|}{1.33E-02}              & 4.14E-02              & \multicolumn{1}{c|}{2.22E-03}              & 1.01E-02              \\ \hline
320                        & \multicolumn{1}{c|}{4.61E-04}              & 2.26E-03              & \multicolumn{1}{c|}{4.24E-05}              & 2.36E-04              \\ \hline
\end{tabular}
\caption{Relative errors of the space-time LRNN-$C^0$DG method when $t=5$ in Example \ref{3ddvwe}}
\label{table3dlrnnc0dgerr}
\end{table}

\begin{table}[]
\centering
\begin{tabular}{|c|cc|cc|}
\hline
{$\tau$}, {$h$}             & \multicolumn{2}{c|}{$5/2$, $1/2$}                                             & \multicolumn{2}{c|}{$5/3$, $1/3$}                                   \\ \hline
\diagbox[width=5em,trim=l]{${\rm DoF}_\sigma$}{Norm} & \multicolumn{1}{c|}{$L^{2}$}           & $H^{1}$ & \multicolumn{1}{c|}{$L^{2}$} & $H^{1}$ \\ \hline
40                         & \multicolumn{1}{c|}{2.29E-01}                        & 4.08E-01              & \multicolumn{1}{c|}{2.62E-01}              & 3.96E-01              \\ \hline
80                         & \multicolumn{1}{c|}{8.90E-02}                        & 1.79E-01              & \multicolumn{1}{c|}{2.75E-02}              & 6.51E-02              \\ \hline
160                        & \multicolumn{1}{c|}{3.63E-02}                        & 6.34E-02              & \multicolumn{1}{c|}{5.67E-03}              & 1.87E-02              \\ \hline
320                        & \multicolumn{1}{c|}{2.33E-03} & 6.74E-03              & \multicolumn{1}{c|}{8.68E-05}              & 4.39E-04              \\ \hline
\end{tabular}
\caption{Relative errors of the space-time LRNN-$C^1$DG method when $t=5$ in Example \ref{3ddvwe}}
\label{table3dlrnnc1dgerr}
\end{table}

\begin{table}[]
\centering
\begin{tabular}{|c|lll|ccc|}
\hline
{$\tau$}  & \multicolumn{3}{c|}{$5/32$}                                                                                            & \multicolumn{3}{c|}{$5/64$}                                              \\ \hline
\diagbox[width=5em,trim=l]{$h$}{Norm} & \multicolumn{1}{c|}{$L^{2}$}                         & \multicolumn{1}{c|}{$H^{1}$}  & \multicolumn{1}{c|}{time}       & \multicolumn{1}{c|}{$L^{2}$}  & \multicolumn{1}{c|}{$H^{1}$}  & time     \\ \hline
1/4           & \multicolumn{1}{c|}{9.52E-03}                        & \multicolumn{1}{c|}{6.54E-02} & 9.19E+01                        & \multicolumn{1}{c|}{9.35E-03} & \multicolumn{1}{c|}{6.54E-02} & 1.91E+02 \\ \hline
1/8           & \multicolumn{1}{c|}{1.57E-03}                        & \multicolumn{1}{c|}{1.71E-02} & 8.21E+02                        & \multicolumn{1}{c|}{1.23E-03} & \multicolumn{1}{c|}{1.72E-02} & 1.66E+03 \\ \hline
1/16          & \multicolumn{1}{c|}{{9.36E-04}} & \multicolumn{1}{c|}{4.36E-03} & {7.48E+03} & \multicolumn{1}{c|}{2.80E-04} & \multicolumn{1}{c|}{4.35E-03} & {1.54E+04} \\ \hline
\end{tabular}
\caption{Relative errors of the IPDG-CN method when $t=5$ and total training time in Example \ref{3ddvwe}}
\label{table3dcndg}
\end{table}

We vary the mesh sizes, time sub-intervals, and numbers of degrees of freedom ${\rm DoF}_\sigma$ and report the relative errors of the three methods at $t=5$ in Tables \ref{table3dlrnndgerr}, \ref{table3dlrnnc0dgerr} and \ref{table3dlrnnc1dgerr}, respectively. We use 9 quadrature points in each dimension for all three methods. For the LRNN-DG method, we use $r = 0.25$, $\beta_1 = \beta_2 = 42$. For the LRNN-$C^0$DG method, we use $r = 0.16$ and 41 collocation points in each volume. For the LRNN-$C^1$DG method, we use $r = 0.17$ and 37 collocation points in each volume. From the three tables, we notice that the LRNN-DG method achieves higher accuracy than the other two methods.

We compare our proposed methods with the method of interior penalty DG and Crank-Nicholson (IPDG-CN). For the IPDG-CN method, we use piecewise $P_2$ polynomial functions on tetrahedral meshes and penalty parameter 27. We report the relative errors and total computing time of the IPDG-CN method for different values of time steps $\tau$ and mesh sizes $h$ in Table \ref{table3dcndg}. The total training time of the LRNN-DG method is also given in Table \ref{table3dlrnndgerr}. All the experiments are run on the same CPU using the FEniCS library (version 2019.1.0) for IPDG-CN method and PyTorch 1.12.1 for the LRNN-DG method. The results show that the LRNN-DG method is more accurate and faster than the IPDG-CN method. For example, the LRNN-DG method achieves a relative $L^2$ error of $2.22\times 10^{-4}$ in 33.5 seconds, while the IPDG-CN method achieves a relative $L^2$ error of $2.80\times 10^{-4}$ in 15400 seconds.

\section{Summary}
\label{summary}

In this paper, we introduce a novel approach to solve the diffusive-viscous wave equation using a hybrid of local randomized neural networks and discontinuous Galerkin methods. The LRNN-DG approach can handle time-dependent problems naturally and efficiently by using a space-time framework, and avoid the error accumulation that occurs in standard finite difference methods for time discretization. We show through numerical experiments that our approach can achieve higher accuracy by either increasing the number of degrees of freedom in each subdomain or decreasing the mesh size and sub-interval. Moreover, we demonstrate that LRNN-DG is very competitive in solving problems that require long-time simulation.

LRNN-DG methods offer promising results, but they also raise many challenges and open questions for future research. Some of these questions are: How can we choose the appropriate activation function for different problems? How can we optimize the selection of the parameters of the hidden layers instead of using a uniform distribution with r? How can we overcome the limitation of the approximation ability of the one-layer model when the number of degrees of freedom grows? Can we use other techniques or models (e.g. deep neural networks) to enhance the methods? How can we provide more rigorous theoretical foundations for our methods? How can we exploit parallel processing to speed up the methods? There are still many aspects that need to be investigated for the proposed methods.

\end{document}